\documentclass[reqno,10pt]{amsart}
\makeatletter
 \usepackage{enumitem}
\newcommand{\Rmnum}[1]{\expandafter\@slowromancap\romannumeral#1@}

\newtheorem{lemma}{Lemma}[section]

\newtheorem{theorem}{Theorem}[section]

\newtheorem{remark}{Remark}[section]
\newtheorem{definition}{Definition}[section]

\numberwithin{equation}{section}

\title[Compressible Euler Equations]
{$L^1$-convergence to generalized Barenblatt solution  for compressible Euler equations  with time-dependent damping}%
\author[ S. Geng,  F. Huang and X. Wu]{  }
\email{sfgeng@xtu.edu.cn}

\email{fhuang@amt.ac.cn}
\email{wuxc19@amss.ac.cn}

\subjclass[2000]{35L65, 76S05, 35K65.}
 \keywords{convergence rates,  compressible Euler Equations,  generalized Barenblatt solution,  time-dependent damping, compensated compacted.}
\begin{document}

\maketitle \centerline{\scshape Shifeng Geng$^{a}$ \ \   Feimin Huang$^{b,c}$ \ \    Xiaochun Wu$^{b}$}
\medskip
{\footnotesize
 \centerline{$^a\ \!$School of Mathematics and Computational Science, Xiangtan University}
 \centerline{Xiangtan $411105$, China}
 \centerline{$^b\ \!$Academy of Mathematics and Systems Science, Chinese Academy of Sciences}
   \centerline{Beijing 100190, China}
 \centerline{$^c\ \!$School of Mathematical Sciences, University of Chinese Academy of Sciences}
 \centerline{Beijing 100049, China}}

\medskip
\bigskip
\begin{abstract}
The large time behavior of entropy solution to the  compressible Euler equations for polytropic gas (the pressure $p(\rho)=\kappa\rho^{\gamma}, \gamma>1$)
with time dependent damping like $-\frac{1}{(1+t)^\lambda}\rho u$ ($0<\lambda<1$) is investigated. By introducing an elaborate iterative method and using the intensive entropy analysis, it is proved that  the $L^\infty$ entropy
solution of compressible Euler equations  with finite initial mass
converges strongly in the natural $L^1$ topology  to a fundamental solution of porous media equation (PME) with time-dependent diffusion, called by generalized Barenblatt solution. 
It is interesting that the $L^1$ decay rate is getting faster and faster as $\lambda$ increases  in $(0, \frac{\gamma}{\gamma+2}]$, while is getting slower and slower   in $[ \frac{\gamma}{\gamma+2}, 1)$.   

\end{abstract}

\section{Introduction}
In this paper, we  are concerned with the  compressible
Euler equations with time-dependent damping as follows:
\begin{align}\label{E:1.1}
\begin{cases}
 \rho_t+(\rho u)_x=0,\\
(\rho u)_t+(\rho u^2+p(\rho))_x=-\frac{1}{(1+t)^{\lambda}}\rho u,
\end{cases}
\end{align}
with finite initial mass
\begin{align}\label{E:1.2}
(\rho,u)(x,0)=(\rho_0,u_0)(x),\hspace{0.2cm}\rho_0(x)\geq 0,
\hspace{0.2cm} \int_{\mathbf{R}}\rho_0(x)\ dx= M>0,
\end{align}
where $\rho, u$ and $p=\kappa\rho^{\gamma}, \kappa=\frac{(\gamma-1)^2}{4\gamma} (1<\gamma<3)$, denote respectively the density, velocity,
and pressure, the damping term $\frac{\mu}{(1+t)^{\lambda}}\rho u$ with physical parameter $  0<\lambda <1$,  is the time-dependent friction effect.  We also use momentum $m=:\rho u$ in what follows for convenience.

For the usual damping model, i.e., $\lambda=0$, due to the damping effect of the frictional force term, the inertial terms in the momentum equation decay to zero faster than the other terms so that the pressure gradient force can be balanced by the damping, stated as Darcy law, see Hsiao and Liu
\cite{Hsiao-Liu-1}. It is conjectured that the system
\eqref{E:1.1}  is time-asymptotically equivalent to the following
decoupled system
\begin{align}\label{E:1.3}
\begin{cases}
 \bar{\rho}_t=\kappa(\bar \rho^{\gamma})_{xx}, \\
\bar{m}=-\kappa (\bar \rho^{\gamma})_{x},
\end{cases}
\end{align}
where $\eqref{E:1.3}_2$ is the famous  Darcy law and $\eqref{E:1.3}_1$ is the porous medium equation (PME),  
which admits a self-similar profile $\bar{\rho}(\frac{x}{\sqrt{1+t}})$ satisfying $\bar{\rho}(\pm\infty)=\rho_\pm>0$, $\rho_+\ne \rho_-$.  The conjecture was first justified in \cite{Hsiao-Liu-1}  when the initial data is a small perturbation 
around the background profile $\bar{\rho}$ and thus away from vacuum. Since then, 
 there have been various works for small initial data scattered in the literature, cf.
\cite{ Hsiao-Liu-2,Hsiao-Luo,  Nishihara, Nishihara-Wang-Yang,Zhao,Zheng}. For large initial data, it is observed in \cite{Liu-Yang-1997, Liu-Yang-2000} that the damping can not prevent the shock formation and thus the weak solution has to be considered. Along this direction, the above conjecture was justified in \cite{Huang-Pan-03,Huang-Pan-06, Serre-Xiao, Zhu-03}. 

As pointed out  by Liu in \cite{Liu}, 
it is very interesting to study the finite mass case, in which the solution is supposed to tend asymptotically to the Barenblatt solution $\bar{\rho}$ with the same mass, which is a fundamental solution of PME. Indeed, Liu in \cite{Liu} constructed a special smooth solution to approximate the Barenblatt solution. The general case was studied in  \cite{Huang-Marcati-Pan} for  $L^\infty$ entropy solution and the convergence rate toward the Barenblatt solution in $L^p\, (p>1)$ norm was also obtained.

Since the compressible
Euler equation is conserved, it is natural to measure the difference between the solution $\rho$
and the Barenblatt solution $\bar{\rho}$  in $L^1$ norm. 
 It is further shown in  \cite{Huang-Pan-Wang} that for $1<\gamma<3$ and any  $\varepsilon>0$, 
\begin{align}\label{E:1.41}
	\|(\rho- \bar{\rho})(\cdot,t)\|_{L^1}\leq
	C(1+t)^{-\frac{1}{4(\gamma+1)}+\varepsilon}.
\end{align}
 Recently, Geng and Huang \cite{Geng-Huang} improved the rate \eqref{E:1.41}  as $2\leq\gamma<3$ and extended the work of \cite{Huang-Pan-Wang} to the case of $\gamma\geq3$.

For the time-dependent damping, i.e., $ 0<\lambda<1$, 
although the damping effect  is getting weaker and weaker as time $t$ increases, the inertial terms in the momentum equation still decay to zero faster than the other terms so that the pressure gradient force can be balanced by the time-dependent damping, and the system \eqref{E:1.1} should be time-asymptotically equivalent to 
\begin{align}\label{E:1.3new}
	\begin{cases}
		\bar{\rho}_t=\kappa(1+t)^{\lambda}(\bar \rho^{\gamma})_{xx}, \\
		\bar{m}=-\kappa(1+t)^{\lambda}(\bar \rho^{\gamma})_{x},
	\end{cases}
\end{align}
where $\eqref{E:1.3new}_1$ is PME with time-dependent diffusion. In fact, it was proved in   \cite{Cui-Yin-Zhang-zhu, Li-Li-Mei-Zhang} that the system \eqref{E:1.1} admits a smooth solution which asymptotically tends to the diffusion wave $\bar\rho$ of $\eqref{E:1.3new}_1$ under some smallness conditions of initial data.  An interesting work for partially large initial data was studied in  \cite{Chen-Li-Li-Mei-Zhang}.  For other related works, we refer to \cite{Pan-1,Pan-2, Sugiyama-1,Sugiyama-2} and reference therein.  

Similar to the case of $\lambda=0$, it is also significant to study the large time behavior of entropy solution of the system \eqref{E:1.1} with finite mass. We will try to obtain the $L^1$ convergence rate of $\|\rho-\bar{\rho}\|_{L^1}$ for  $L^\infty$ weak entropy
solution $\rho$ to  \eqref{E:1.1} without any smallness restriction of initial data, where $\bar\rho$ is the fundamental solution of $\eqref{E:1.3new}_1$, denoted by generalized Barenblatt solution. 

Before stating the main work, let us first recall the definition of $L^\infty$ weak entropy solution to \eqref{E:1.1}.

\begin{definition}\label{D:1.4}
$(\rho, m)(x,t)\in L^{\infty}$ is called an entropy solution of
\eqref{E:1.1} and \eqref{E:1.2}, if for any non-negative test
function $\phi \in \mathcal{D}(\mathbf{R_{+}})$, it holds that
\begin{align}\label{E:1.5}
\begin{cases}
\int\int_{t>0} (\rho\phi_t+m\phi_x)\ dxdt+\int_{\mathbf{R}} \rho_0(x)\phi(x,0)\ dx=0,\\
\int\int_{t>0} [  m\phi_t+(\frac{m^2}{\rho}+p(\rho))\phi_x- \frac{1}{(1+t)^{\lambda}}m\phi]\
dxdt+\int_{\bf{R}} m_0(x)\phi(x,0)\ dx=0,{\tiny }\nonumber
\end{cases}
\end{align}
and
\begin{equation}\label{E:1.51}
 \eta_{t}+q_{x}+ \frac{ \eta_mm}{(1+t)^{\lambda}} \leq0
\end{equation}
\end{definition}
\noindent in the sense of distributions, where  $(\eta, q)$ is any
weak convex entropy-flux pair $(\eta(\rho,m), q(\rho,m))$ satisfying
\begin{equation}\label{E:1.51-1}
 \nabla q=\nabla \eta \nabla f, \hspace{0.2cm} f=\Big(m, \frac{m^2}{\rho}+\kappa\rho^\gamma\Big)^t, \hspace{0.2cm} \eta(0,0)=0.
\end{equation}


Without the damping term, i.e., the right hand side of \eqref{E:1.1} is zero, Diperna \cite{Diperna} first proved the global existence of $L^\infty$ entropy solutions with arbitrarily large initial data by the theory of compensated compactness 
  for $\gamma=1+\frac{2}{2n+1}$, with any integer $n \geq 2$. Subsequently, Ding et al. \cite{Ding-Chen-Luo} and Chen \cite{Chen} successfully extended the result to $\gamma\in(1, \frac{5}{3}]$. Lions et al. \cite{Lions-Perthame-Tadmor} and \cite{Lions-Perthame-Souganidis} treated the 
  case $\gamma>\frac{5}{3}$. Following the arguments of \cite{Chen,Ding-Chen-Luo,Diperna,Lions-Perthame-Tadmor,Lions-Perthame-Souganidis}, it is not difficult to prove the global existence of entropy solution to the system \eqref{E:1.1} for any bounded initial data, see also \cite{Huang-Pan-03}.

\

The precise statement of  our main result is
\begin{theorem}\label{theorem 1.1}
Suppose that $\rho_0(x)\in L^1(\mathbf{R})\cap
L^\infty(\mathbf{R}), u_0(x)\in L^\infty(\mathbf{R})$ and
\begin{equation} M=\int_{\mathbf{R}} \rho_0(x)\ dx>0.\nonumber
\end{equation}
Let $\lambda\in(0,1)$, $\gamma \in (1,3)$ and $(\rho,m)$  be the $L^{\infty}$ entropy
solution of the Cauchy problem $\eqref{E:1.1},\eqref{E:1.2}$. Let
$\bar{\rho}$ be the generalized Barenblatt solution of 
$\eqref{E:1.3new}_1$ with mass $M$. Let \begin{equation}
\label{E:1.8} y=-\int_{-\infty}^{x}(\rho- \bar{\rho})(r,t)\ dr.
\nonumber
\end{equation}
If $y(x,0)\in L^2(\mathbf{R})$, then for any $\varepsilon >0$   and
$t>0$, it holds that
 \begin{align}\label{E:1.9}
\|(\rho- \bar{\rho})(\cdot,t)\|_{L^{\gamma+1}}^{\gamma+1}\leq C(1+t)^{-\mu(\varepsilon)},
\end{align}
where 
\begin{equation*}
\mu(\varepsilon)=
\left\{
\begin{array}{ll}
1+\lambda-\frac{\lambda+1}{2(\gamma+1)}-\varepsilon, &\quad \lambda\in(0,\frac{\gamma}{\gamma+2}],\\
\frac{3}{2}+\frac{\lambda}{2}-\frac{\lambda+1}{\gamma+1}-\varepsilon, &\quad \lambda \in [\frac{\gamma}{\gamma+2},1).
\end{array}
\right.
\end{equation*}
Furthermore, 
 \begin{align}\label{E:1.9-1}
\|(\rho- \bar{\rho})(\cdot,t)\|_{L^1}\leq C(1+t)^{-\alpha(\varepsilon)},
\end{align}
with
	\begin{equation}\label{1.11}
\alpha(\varepsilon)=
\left\{
\begin{array}{ll}
\frac{\lambda+1}{4(\gamma+1)}-\varepsilon, &\quad \lambda\in(0,\frac{\gamma}{\gamma+2}]\\
\frac{1-\lambda}{4}-\varepsilon, &\quad \lambda \in [\frac{\gamma}{\gamma+2},1).\\
\end{array}
\right.
\end{equation}
\end{theorem}
\begin{remark}
	Theorem 1.1 shows that any $L^\infty$ entropy solution of  \eqref{E:1.1} satisfying
	the conditions in Theorem \ref{theorem 1.1} must converge to the generalized  Barenblatt
	solution of \eqref{E:1.3new} with the same mass. Although there is no uniqueness for
	the solutions, Theorem \ref{theorem 1.1} indicates the unique asymptotic profile of solution, determined
	by the initial mass.   
\end{remark}
\begin{remark} In terms of \eqref{E:1.9-1} and \eqref{1.11}, 
    the $L^1$ decay rate is getting faster and faster as $\lambda$ increases  in $(0, \frac{\gamma}{\gamma+2}]$, while is getting slower and slower   in $[ \frac{\gamma}{\gamma+2}, 1)$.  
\end{remark}

 We now sketch the main strategy.  Different from the case of $\lambda=0$, the main difficulties come from the interaction of pressure,  time-dependent diffusion  and vacuum.  In fact,  it was observed in  \cite{Geng-Lin-Mei}  for  $\lambda=1$ that the solution to  $\eqref{E:1.3new}_1$ is no longer the asymptotic profile of the solutions to system \eqref{E:1.1}.   Since the damping is getting weaker and weaker as $t$ increases, the methods developed in  \cite{Geng-Huang,Huang-Marcati-Pan, Huang-Pan-Wang} do not work anymore.   More precisely,
a nonlinear wave equation \eqref{E:2.12} for $y$ is applied to get 
 \begin{align}\label{b1}
 	\int_{\mathbf{R}}\frac{ y^2}{(1+t)^{\lambda}}dx+\int_0^t\int_{\mathbf{R}}|\rho-\bar\rho|^{\gamma+1}dxd\tau\leq C+C\int_0^t\int_{\mathbf{R}}\frac{m^2}{\rho}\bar{\rho}dxd\tau+C(1+t)^{\lambda}\int_{\mathbf{R}}y_t^2dx.
 \end{align}
  As $\lambda=0$,  the two terms on the right hand side of \eqref{b1} are controlled 
 by the energy inequality 
 \begin{eqnarray}\label{energy}
 \int_\mathbf{R}\eta_e dx+\int_0^t\int_\mathbf{R} \frac{m^2}{\rho}dxd\tau\le C,
 \end{eqnarray}
see \cite{Huang-Pan-Wang}, where $\eta_e=\frac{m^2}{2}+\frac{\kappa}{\gamma-1}\rho^\gamma$ is the mechanical energy. When $\frac{1}{\gamma}<\lambda<1$, \eqref{energy} becomes 
 \begin{eqnarray}\label{newenergy}
 \int_\mathbf{R}\eta_e dx+\int_0^t\int_\mathbf{R} (1+\tau)^{-\lambda}\frac{m^2}{\rho}dxd\tau\le C,
 \end{eqnarray}
 so that the terms $\int_0^t\int_{\mathbf{R}}\frac{m^2}{\rho}\bar{\rho}dxd\tau$ and $(1+t)^{\lambda}\int_{\mathbf{R}}y_t^2dx$ in \eqref{b1} are difficult to be estimated. We will use two kinds of entropies $\eta_e$ and $\tilde \eta$ (see \eqref{E:2.9-1}) to deal with the two terms. By the energy inequality and a careful decomposition of the integral region, we can get
 \begin{align}\label{w-0-0}
		\int_0^t\int_{\bf{R}}(1+\tau)^{a-\lambda}\frac{m^2}{\rho}dxd\tau
		\leq C+\nu_1\int_0^t\int_{\bf{R}}|\rho-\bar\rho|^{\gamma+1}dxd\tau,\quad \forall\, a\in(0,1-\frac{1}{\gamma}),
	\end{align} 
where $\nu_1$ is a small constant, see Lemma \ref{lemma-1}. When $a$ is close to $1-\frac{1}{\gamma}$, the term $\int_0^t\int_{\mathbf{R}}\frac{m^2}{\rho}\bar{\rho}dxd\tau$ is controlled by $\int_0^t\int_{\bf{R}}(1+\tau)^{a-\lambda}\frac{m^2}{\rho}dxd\tau$, which can be absorbed into the left hand side of \eqref{b1} due to \eqref{w-0-0}. In the same argument,  $(1+t)^{\lambda}\int_{\mathbf{R}}y_t^2dx$ can be estimated by using another entropy $\tilde \eta$, see Lemma \ref{lemma-2}, so that the basic estimates 
	\begin{align}\label{w-3.36-0}
		&(1+t)^{\lambda}\int_{\mathbf{R}}m^2dx+(1+t)^{\lambda}\int_{\mathbf{R}}|\rho-\bar \rho|^{\gamma+1} dx
	+\int_0^t\int_{\mathbf{R}}|\rho-\bar\rho|^{\gamma+1}dxd\tau \leq C
			\end{align}
are derived in Lemma \ref{lemma3.5}. 

Since the generalized Barenblatt solution $\bar \rho$ is compactly supported, the rate of $\|\rho-\bar \rho\|_{L^{1}}$ can be transformed by a fast decay rate of $\|\rho-\bar \rho\|_{L^{\gamma+1}}$. However, the rate $(1+t)^{-\lambda}$ of $\|\rho-\bar \rho\|_{L^{\gamma+1}}$ in \eqref{w-3.36-0} is not fast enough.  Note that a relative entropy $\eta_{\ast}$ for $\tilde \eta$ can be applied to get a better decay rate of $\|\rho-\bar \rho\|_{L^{\gamma+1}}$ as in \cite{Huang-Pan-Wang},
	\begin{align}\label{3.41-0}
		&(1+t)^{\mu_1}\int_{\mathbf{R}}m^2dx+(1+t)^{\mu_1}\int_{\mathbf{R}}|\rho-\bar \rho|^{\gamma+1}dx+	\int_0^t\int_{\mathbf{R}}(1+\tau)^{\mu_1-\lambda}y_t^2dxd\tau\leq C,
		\end{align}
		where $\mu_1>\lambda$.
Although the rate of $\|\rho-\bar \rho\|_{L^{\gamma+1}}$ is improved, it is still not fast enough to achieve the goal. To get the desired rate of $\|\rho-\bar \rho\|_{L^{\gamma+1}}$, 
an elaborate iterative method for the rate of $\|\rho-\bar \rho\|_{L^{\gamma+1}}$ is introduced. In fact,   \eqref{3.41-0} provides two new estimates for $(1+t)^{\mu_1}\int_{\mathbf{R}}m^2dx$ and $\int_0^t\int_{\mathbf{R}}(1+\tau)^{\mu_1-\lambda}y_t^2dxd\tau$, which, together with the wave equation \eqref{E:2.12}, imply the following inequality,
	\begin{align}\label{3.49-0}
		\int_0^t\int_{\mathbf{R}}(1+\tau)^{\theta_1}|\rho-\bar \rho|^{\gamma+1}dxd\tau \leq C,\quad \theta_1 >0.
		\end{align}
Note that \eqref{3.49-0}  is better than that in \eqref{w-3.36-0}. With the help of  the relative entropy method, we can further improve  \eqref{3.41-0} to 
\begin{align}\label{3.41-0-0}
		&(1+t)^{\mu_2}\int_{\mathbf{R}}m^2dx+(1+t)^{\mu_2}\int_{\mathbf{R}}|\rho-\bar \rho|^{\gamma+1}dx+	\int_0^t\int_{\mathbf{R}}(1+\tau)^{\mu_2-\lambda}y_t^2dxd\tau\leq C, 
		\end{align}
where $\mu_2>\mu_1$ is a constant. Repeating the same procedure by finitely many times,  finally we can get the desired decay rate of $\|\rho-\bar{\rho}\|_{L^{\gamma+1}}$, from which the  rate of $\|\rho-\bar \rho\|_{L^{1}}$ in Theorem \ref{theorem 1.1} is obtained.

\

The arrangement of the present paper is as follows. In Section 2,
 the generalized Barenblatt solution is constructed and its explicit formula is given. The $L^1$ decay estimates are obtained in Section 3.

\

\noindent{\bf Notations}. \ \ Throughout this paper the symbol $C$
 will be used to represent a generic constant which is
independent of $x$ and $t$ and may vary from line to line.
$\|\cdot\|_{L^p}$ stands for the
$L^p(\mathbf{R})$-norm $(1\leq p\leq \infty)$. The $L^2$-norm on $\mathbf{R}$ is simply
denoted by $\|\cdot\|$. Moreover, the domain $\mathbf{R}$ will be
often abbreviated without confusions.

\section{Generalized Barenblatt Solutions}

In this section, we first give the explicit formula of the generalized Barenblatt solution, which is the fundamental solution  of the following porous medium equation with time-dependent damping
\begin{align}\label{E:1.31}
	\begin{cases}
		\bar{\rho}_t=\kappa(1+t)^{\lambda}( \bar{\rho}^{\gamma})_{xx}, \\
		\bar{\rho}(0,x)=M\delta (x),\hspace{0.5cm} M>0.
	\end{cases}
\end{align}

\begin{lemma} \label{lemma2.1}
There exists one and only one solution $\bar{\rho}(x,t)$ to \eqref{E:1.31} taking the form:
\begin{equation} \label{E:2.2}
\bar \rho=(1+t)^{-\frac{\lambda+1}{\gamma+1}}\{(A-B\xi^2)_+\}^{\frac{1}{\gamma-1}},
\end{equation}
where $\xi=x(1+t)^{-\frac{\lambda+1}{\gamma+1}}, (f)_{+}=\max\{0,f\},
B=\frac{(\lambda+1)(\gamma-1)}{2\kappa \gamma(\gamma+1)}$ and $A$ is given by
\begin{equation} \label{E:2.3}
M=2\sqrt{\frac{A}{B}}A^{\frac{1}{\gamma-1}}\int_0^1(1-y^2)^{\frac{1}{\gamma-1}}dy.\nonumber
\end{equation}
Furthermore, 
\begin{itemize} \item [(1)] $\ \bar{\rho}(x,t)$ is continuous on $\mathbf{R}$.
 \item [(2)] There is a number $b=\sqrt{\frac{A}{B}}>0$, such that  $\bar{\rho}(x,t)>0$ if $|x|<b(1+t)^{\frac{\lambda+1}{\gamma+1}}$; and $\bar{\rho}(x,t)=0$ if $|x|\geq bt^{\frac{\lambda+1}{\gamma+1}}$.
 \item[(3)] $ \bar{\rho}(x,t)$ is smooth if $|x|<bt^{\frac{\lambda+1}{\gamma+1}}$.
\end{itemize} \end{lemma}

\begin{remark}
Since the derivative of $\bar{\rho}$ is
	not continuous across the interface between the gas and vacuum, $\bar{\rho}$ is understood as a weak solution to \eqref{E:1.31}, see \cite{Aronson,Barenblatt,Brezis-Crandall} for the definition of weak solution.
\end{remark}
\begin{proof} Following \cite{Barenblatt}, we seek for the solution $\bar\rho$ in the form of
	\begin{equation}\label{2.6}
	\bar \rho=(1+t)^{-s}f(\xi),
	\end{equation}
where $f$ is a smooth function with a compact support, $\xi=x(1+t)^{-s}$ and $s$ is a positive constant to be determined. Substituting \eqref{2.6} into \eqref{E:1.31} implies 
	\begin{equation}\label{2.7}
	s(1+t)^{-s-1}(f+\xi f_{\xi})+(1+t)^{\lambda}(1+t)^{-s(\gamma+2)}(\kappa f^{\gamma})_{\xi\xi}=0.
	\end{equation}	
Then we have
	\[ -s-1=-s(\gamma+2)+\lambda,\]
	which gives $s=\frac{\lambda+1}{\gamma+1}$ and 
	\begin{align*}
	 \Big(\frac{\lambda+1}{\gamma+1}\xi f+\kappa f^{\gamma}_{\xi}\Big)_{\xi}=0.
	\end{align*}
We deduce from the compact support of $f(\xi)$ that 
	 \begin{align*}
 f^{\gamma-1}=A-B\xi^2,
	\end{align*}
	where \[B=\frac{(\lambda+1)(\gamma-1)}{2\kappa \gamma(\gamma+1)}.\]	
Thus we conclude that
	\begin{equation}\label{2}
	\bar \rho=(1+t)^{-\frac{\lambda+1}{\gamma+1}}\{(A-B\xi^2)_+\}^{\frac{1}{\gamma-1}},\quad \xi=\frac{x}{(1+t)^\frac{\lambda+1}{\gamma+1}},
	\end{equation}
	which concides with the Barenblatt solution as $\lambda=0$.
On the other hand, the mass conservation \eqref{E:1.31} implies that 
	\[\int \bar \rho dx=\int \bar \rho(0,x)dx=M,\]
	which urges us to specify the value of $A$ by
	\begin{align*}
		M=2\int_0^{\sqrt{\frac{A}{B}}}(A-Bx^2)^{\frac{1}{\gamma-1}}dx
		=2\sqrt{\frac{A}{B}}A^{\frac{1}{\gamma-1}}\int_0^1(1-y^2)^{\frac{1}{\gamma-1}}dy.
	\end{align*}
In terms of the explicit form of $\bar{\rho}$, it is straightforward to check that  the properties (1)-(3)  hold. Thus, the proof of Lemma \ref{lemma2.1} is completed.
	\end{proof}

Furthermore, we list some properties of $\bar \rho$ as follows, which can be directly obtained by the explicit formula \eqref{E:2.2}.
\begin{lemma}\label{lemma2.2}
It holds that
\begin{align}\label{E:2.6}
\begin{split}
 &\|\bar{\rho}(\cdot,t)\|_{L^\infty}\leq C(1+t)^{-\frac{\lambda+1}{\gamma+1}},\\
 &\|(\bar{\rho}^{\gamma-1})_x(\cdot,t)\|_{L^\infty}\leq C(1+t)^{-\frac{(\lambda+1)\gamma}{\gamma+1}},\hspace{0.5cm}\|(\bar{\rho}^{\gamma-1})_t(\cdot,t)\|_{L^\infty}\leq C(1+t)^{-\frac{2\gamma+\lambda(\gamma-1)}{\gamma+1}},\\
&\|(\bar{\rho}^{\gamma})_x(\cdot,t)\|_{L^\infty}\leq
C(1+t)^{-1-\lambda},\hspace{0.5cm}\|(\bar{\rho}^{\gamma})_t(\cdot,t)\|_{L^\infty}\leq
C(1+t)^{-\frac{2\gamma+1+\lambda \gamma}{\gamma+1}},
\end{split}
\end{align}
and
\begin{align}\label{E:2.7}
\begin{split}
 &\|\bar{\rho}(\cdot,t)\|_{L^p}\leq C(1+t)^{-\frac{(\lambda+1)(p-1)}{p(\gamma+1)}}, \ \forall p\geq 1,\\
 &\|(\bar{\rho}^{\gamma-1})_x(\cdot,t)\|\leq C(1+t)^{-\frac{(\lambda+1)(2\gamma-1)}{2(\gamma+1)}},\hspace{0.3cm}\|(\bar{\rho}^{\gamma-1})_t(\cdot,t)\|\leq C(1+t)^{-\frac{4\gamma-1+(2\gamma-3)\lambda}{2(\gamma+1)}},\\
&\|(\bar{\rho}^{\gamma})_x(\cdot,t)\|\leq
C(1+t)^{-\frac{(2\gamma+1)(\lambda+1)}{2(\gamma+1)}},\hspace{0.3cm}\|(\bar{\rho}^{\gamma})_t(\cdot,t)\|\leq
C(1+t)^{-\frac{4\gamma+1+(2\gamma-1)\lambda}{2(\gamma+1)}}.
\end{split}
\end{align}
\end{lemma}

\section{$L^1$ Convergence Estimates}

This section is devoted to the $L^1$ decay rate. We first give an invariant region theorem for
$L^\infty$ entropy solution to \eqref{E:1.1}, \eqref{E:1.2} for $\lambda=0$ due to \cite{Huang-Pan-Wang}. 
\begin{theorem} [\cite{Huang-Pan-Wang}]\label{theorem2.1}
Let $\lambda=0$, assume that $(\rho_0, u_0)(x)\in L^\infty(\mathbf{R})$ satisfies
\begin{equation}\label{E:2.13}
0\leq \rho_0(x)\leq C,\hspace{0.3cm}| m_0(x)|\leq
C\rho_0(x).\nonumber
\end{equation}
Let $(\rho, u)\in L^\infty(\mathbf{R}\times[0,T])$ be the $L^\infty$
 entropy solution of the system  \eqref{E:1.1}, \eqref{E:1.2}
with $\gamma>1$. Then $(\rho, m)$ satisfies
\begin{equation*} 
0\leq \rho(x,t) \leq C, \hspace{0.3cm} |m(x,t)|\leq C\rho(x,t),
\end{equation*}
where the constant $C$ depends solely on the initial data.
\end{theorem}

\begin{remark}
	By the same argument of \cite{Huang-Pan-Wang}, Theorem \ref{theorem2.1} still holds for $\lambda>0$.
\end{remark}

Suppose that
$(\rho,m)$ is an entropy solution of \eqref{E:1.1}, \eqref{E:1.2}
satisfying the conditions in Theorem \ref{theorem 1.1}, then  $(\rho,m)$
satisfies
\begin{align}\label{E:2.7}
\begin{cases}
 \rho_t+m_x=0,\\
m_t+(\frac{ m^2}{\rho}+\kappa \rho^{\gamma})_x=-\frac{1}{(1+t)^{\lambda}}m.
\end{cases}
\end{align}
Let $\bar{\rho}$ be the Barenblatt solution of porous medium
equation \eqref {E:1.31} as in Lemma \ref{lemma2.1} carrying the same total mass $M$
 as $\rho$, and $\bar m=-\kappa (1+t)^{\lambda}\bar \rho^{\gamma}_x$. Then $(\bar{\rho}, \bar m)$ satisfies 
\begin{align}\label{E:2.1}
	\begin{cases}
		\bar{\rho}_t+\bar m_x=0, \\
		\kappa \bar \rho^{\gamma}_x=-\frac{1}{(1+t)^{\lambda}}\bar m.
	\end{cases}
\end{align}
Define
\begin{align}
w=\rho-\bar{\rho},\ \ z=m-\bar{m},\nonumber
\end{align}
which satisfies
\begin{align}\label{E:2.9}
\begin{cases}
 w_t+z_x=0,\\
z_t+(\frac{ m^2}{\rho})_x+\kappa (\rho^{\gamma}-\bar \rho^{\gamma})_x+\frac{1}{(1+t)^{\lambda}}z=- \bar{m}_t.
\end{cases}
\end{align}
Setting
\begin{equation}
y=-\int_{-\infty}^xw(r,t)\ dr,\nonumber
\end{equation}
we have
\begin{equation}
y_x=-w, \hspace{0.5cm} z=y_t.\nonumber
\end{equation}
Thus the equation \eqref{E:2.9} turns into a nonlinear wave equation
with source term, which is degenerate at vacuum:
\begin{equation}\label{E:2.12}
y_{tt}+(\frac{ m^2}{\rho})_x+\kappa (\rho^{\gamma}-\bar{\rho}^{\gamma})_x+\frac{1}{(1+t)^{\lambda}}y_t=-
\bar{m}_t.
\end{equation}

Before deriving the basic energy estimates, we give some important inequalities which provide sharp
information on the pressure near vacuum as below.
\begin{lemma} \label{lemma2.3}$($\cite{Huang-Marcati-Pan,
		Huang-Pan-Wang}$)$ Let $0\leq \rho,
	\bar{\rho}\leq C,$ there are two constants $c_1>0$ and  $c_2>0$ such
	that
	\begin{align*}
		\begin{split}
			(\rho^{\gamma}-\bar{\rho}^{\gamma})(\rho-\bar{\rho})&\geq |\rho-\bar{\rho}|^{\gamma+1}, \\
			c_1(\rho^{\gamma-1}+\bar{\rho}^{\gamma-1})(\rho-\bar{\rho})^2&\leq    \rho^{\gamma+1}-\bar \rho^{\gamma+1}-(\gamma+1)\bar \rho^{\gamma}(\rho-\bar{\rho})\\
			&\leq c_2(\rho^{\gamma-1}+\bar{\rho}^{\gamma-1})(\rho-\bar{\rho})^2,\\
			c_1(\rho^{\gamma-1}+\bar{\rho}^{\gamma-1})(\rho-\bar{\rho})^2&\leq(\rho^{\gamma}-\bar{\rho}^{\gamma})(\rho-\bar{\rho})\\
			&\leq c_2(\rho^{\gamma-1}+\bar{\rho}^{\gamma-1})(\rho-\bar{\rho})^2.
	\end{split}\end{align*}
\end{lemma}

The remainder of this section will be devoted to the proof of Theorem \ref{theorem 1.1} and the proof is devided into three subsections. 

\subsection{Basic energy estimates}

According to \cite{Lions-Perthame-Tadmor}, all weak entropies of \eqref{E:1.51} are given as follows:
\begin{align*}
	\eta(\rho, u)&=\int g(\xi)\chi(\xi; \rho, u)d\xi=\rho\int_{-1}^1g(u+z\rho^{\theta})(1-z^2)^{l}dz,\\
	q(\rho,u)&=\int g(\xi)(\theta\xi+(1-\theta)u)\chi(\xi; \rho, u)d \xi \nonumber\\
	&=\rho\int_{-1}^1g(u+z\rho^{\theta})(u+\theta z\rho^{\theta})(1-z^2)^{l}dz,
\end{align*}
where $\theta=\frac{\gamma-1}{2}$, $l=\frac{3-\gamma}{2(\gamma-1)}$, $g(\xi)$ is any smooth function of $\xi$, and 
\[\chi(\xi;\rho, u)=(\rho^{\gamma-1}-(\xi-u)^2)_+^{l}.\]

To estimate $\|\rho-\bar{\rho}\|_{L^{\gamma+1}}$, we will use two kinds  of entropies by choosing $g(\xi)=\frac{1}{2}\xi^2$ and $g(\xi)=|\xi|^{\frac{2\gamma}{\gamma-1}}$ respectively. It is noted that when $g(\xi)=\frac{1}{2}\xi^2$, the entropy becomes mechanical energy
\begin{equation}\label{E:2.8-1}
\eta_e=\frac{m^2}{2\rho}+\frac{\kappa}{\gamma-1}\rho^{\gamma};
\end{equation}
when  $g(\xi)=|\xi|^{\frac{2\gamma}{\gamma-1}}$, the entropy reads as 
\begin{equation}\label{E:2.8-2}
\tilde{\eta}=\rho\int_{-1}^1|u+z\rho^{\theta}|^{\frac{2\gamma}{\gamma-1}}(1-z^2)^{l}dz\nonumber,
\end{equation}
which is used to measure the $L^{\gamma+1}$ norm  since the highest power  of $
\tilde \eta$ for density $\rho$ is $\gamma+1$.
By Taylor's expansion as in \cite{Huang-Pan-Wang}, we have
\begin{equation}\label{E:2.9-1}
\tilde{\eta}=C_1\rho^{\gamma+1}+C_2m^2+A(\rho,m),
\end{equation}
where 
\begin{align}
	A(\rho, m)=&\rho \int_{-1}^1\Big(g(u+z\rho^{\theta})-g(z\rho^{\theta})-g^{\prime}(z\rho^{\theta})u-\frac{1}{2}g^{\prime\prime}(z\rho^{\theta})u^2\Big)(1-z^2)^{l}dz,\nonumber\\
	=&\rho u^3\int_{-1}^1\int_0^1\frac{(1-s)^2}{2}g^{(3)}(su+z\rho^{\theta})(1-z^2)^{l}dsdz,\label{3.8}\\
 C_1=&\int_{-1}^1|z|^{\frac{2\gamma}{\gamma-1}}(1-z^2)^{\lambda}dz=\frac{1}{2}B(\frac{\gamma+1}{2(\gamma-1)},\frac{\gamma+1}{2(\gamma-1)}), \nonumber\\
	C_2=&\frac{\gamma(\gamma+1)}{(\gamma-1)^2}\int_{-1}^1|z|^{\frac{2}{\gamma-1}}(1-z^2)^{\lambda}dz=\frac{2\gamma(\gamma+1)}{(\gamma-1)^2}C_1\label{E:2.8-3}
\end{align}
and $B(p,q)$ is a Beta function defined by
\begin{equation}\label{E:2.8-4}
B(p,q)=\int_0^1x^{p-1}(1-x)^{q-1}dx.\nonumber
\end{equation}
From \eqref{E:1.51}, it holds that
\begin{equation}\label{en.2}
\eta_{et}+q_{ex}+\frac{ m^2}{(1+t)^{\lambda}\rho}\leq 0
\end{equation}
and 
\begin{equation}\label{en.1}
\tilde{\eta}_t+ \tilde{q}_x+\frac{2C_2 m^2}{(1+t)^{\lambda}}+\frac{A_mm}{(1+t)^{\lambda}}\leq 0
\end{equation}
in the sense of distributions, where $q_e$ and $\tilde{q}$ are the corresponding fluxes. 
Thus, by the theory of divergence-measure fields (see Chen and Frid \cite{Chen-Frid}),
we have
\begin{equation}\label{3.15}
\int_{\mathbf{R}}\eta_e dx+ \int_0^t\int_{\mathbf{R}}\frac{m^2}{(1+\tau)^{\lambda}\rho}dxd\tau\leq C
\end{equation}
and 
\begin{equation}\label{3.16}
\int_{\mathbf{R}}\tilde{\eta} dx+2C_2 \int_0^t\int_{\mathbf{R}}\frac{m^2}{(1+\tau)^{\lambda}}dxd\tau+ \int_0^t\int_{\mathbf{R}}\frac{A_mm}{(1+\tau)^{\lambda}}dxd\tau\leq C.
\end{equation}

\begin{lemma}\label{lemma3.2}
	For $\gamma\in(1,3)$, it holds that
	\begin{align*}
		A(\rho, m)\geq 0 \quad\mbox{and}\quad A_mm\geq3A(\rho,m)\geq 0.
	\end{align*}
\end{lemma}
\begin{proof}
	By direct computations, $A(\rho,m)$ and $A_mm$ can be rewritten as 
	\begin{align*}
		A(\rho, m)=\rho u^4 \int_{-1}^1\int_{0}^1\int_{0}^1\frac{(1-s_1)^2s_2}{2}g^{(4)}(s_1s_2u+z\rho^{\theta})(1-z^2)^{l}d{s_1}d{s_2}dz\geq 0
	\end{align*}
	and 
	\begin{align*}
		A_mm=3A+\rho u^4 \int_{-1}^1\int_{0}^1\frac{(1-s)^2s}{2}g^{(4)}(su+z\rho^{\theta})(1-z^2)^{l}dsdz\geq 3A,
	\end{align*}
	where we have used the fact $g^{(4)}\geq 0$ for $\gamma\in(1,3)$. Thus, the proof of the Lemma \ref{lemma3.2} is completed.
\end{proof}

\begin{lemma}\label{lemma-2}
	Under the conditions of Theorem \ref{theorem 1.1}, it holds that 
	\begin{align}\label{w-3.21}
		&(1+t)^{\lambda}\int_{\bf{R}}\tilde{\eta} dx+\int_{0}^t\int_{\bf{R}}m^2dxd\tau+\int_{0}^t\int_{\bf{R}}A(\rho, m)dxd\tau
		\notag\\
		\leq& C+C\lambda\nu_1^{\frac{\gamma(1-\lambda)}{2\lambda}}\int_{0}^t\int_{\bf{R}}(\rho^{\gamma}-\bar \rho^{\gamma})(\rho-\bar\rho)dxd\tau,
		\end{align}
where $\nu_1$ is a positive constant to be determined.
	\end{lemma}
\begin{proof}
		Multiplying $(1+t)^{\lambda}$ with \eqref{en.1} and integrating over $\mathbf{R}\times [0,t]$  give
	\begin{align}\label{w-4}
		&(1+t)^{\lambda}\int_{\bf{R}}\tilde{\eta} dx+2C_2\int_0^t\int_{\bf{R}}m^2dxd\tau+\int_0^t\int_{\bf{R}}A_mmdxd\tau
		\notag\\
		\leq &C_0+	\lambda C_1\int_0^t\int_{\bf{R}}(1+\tau)^{\lambda-1}\rho^{\gamma+1}dxd\tau+\lambda C_2\int_0^t\int_{\bf{R}}(1+\tau)^{\lambda-1}m^2dxd\tau
		\notag\\
		&+\lambda\int_0^t\int_{\bf{R}}(1+\tau)^{\lambda-1}A(\rho,m)dxd\tau.
	\end{align}
Since $A_mm\geq 3A$ and $\lambda\in(0,1)$, we only need to consider the second term on the right hand side of \eqref{w-4}, which is estimated by deviding the integral region into two parts as follows:
\begin{align}\label{w-3.23}
		&\int_0^t\int_{\bf{R}}(1+\tau)^{\lambda-1}\rho^{\gamma+1}dxd\tau
		\notag\\
		=& \int_0^t\int_{\{x:0\leq \rho\leq \frac{(1+\tau)^{-2\lambda\gamma^{-1}}}{\nu_1}\}}(1+\tau)^{\lambda-1}\rho^{\gamma+1}dxd\tau+\int_0^t\int_{\{x:\rho> \frac{(1+\tau)^{-2\lambda\gamma^{-1}}}{\nu_1}\}}(1+\tau)^{\lambda-1}\rho^{\gamma+1}dxd\tau
		\notag\\
		=&:I_1+I_2.
	\end{align}	
A direct computation implies
\begin{align}\label{w-3.24}
	I_1&=\int_0^t\int_{\{x:0\leq \rho\leq \frac{(1+\tau)^{-2\lambda\gamma^{-1}}}{\nu_1}\}}(1+\tau)^{\lambda-1} \rho^{\gamma+1}dxd\tau
\leq \frac{M}{\nu_1^{\gamma}}\int_0^t(1+\tau)^{-\lambda-1}d\tau
	\leq C
	\end{align}
and
\begin{align}\label{w-3.25}
	I_2&=\int_0^t\int_{\{x:\rho> \frac{(1+\tau)^{-2\lambda\gamma^{-1}}}{\nu_1}\}}(1+\tau)^{\lambda-1}\rho^{\gamma+1}dxd\tau
	\notag\\
	&\leq 2^{\gamma+1}\int_0^t\int_{\{x:\rho> \frac{(1+\tau)^{-2\lambda\gamma^{-1}}}{\nu_1}\}}(1+\tau)^{\lambda-1}|\rho-\bar \rho|^{\gamma+1}dxd\tau
	+2^{\gamma+1}\int_0^t\int_{\mathbf{R}}(1+\tau)^{\lambda-1}\bar \rho^{\gamma+1}dxd\tau
	\notag\\
	&\leq C\nu_1^{\frac{\gamma(1-\lambda)}{2\lambda}}\int_0^t\int_{\mathbf{R}}\rho^{\frac{\gamma(1-\lambda)}{2\lambda}}|\rho-\bar \rho|^{\gamma+1}dxd\tau
+C\int_0^t(1+\tau)^{-2+\frac{\lambda+1}{\gamma+1}}d\tau
	\notag\\
	&\leq C\nu_1^{\frac{\gamma(1-\lambda)}{2\lambda}}\int_0^t\int_{\mathbf{R}}|\rho-\bar \rho|^{\gamma+1}dxd\tau
	+C,
	\end{align}
where we have used the fact $-2+\frac{\lambda+1}{\gamma+1}<-1$ due to $\lambda\in(0,1)$.
Substituting \eqref{w-3.24}-\eqref{w-3.25} into \eqref{w-3.23}, we get
	\begin{align*}
		\int_0^t\int_{\bf{R}}(1+\tau)^{\lambda-1}\rho^{\gamma+1}dxd\tau
	\leq C\nu_1^{\frac{\gamma(1-\lambda)}{2\lambda}}\int_0^t\int_{\mathbf{R}}|\rho-\bar \rho|^{\gamma+1}dxd\tau
		+C,
	\end{align*}
which yields \eqref{w-3.21}.
Thus, the proof of Lemma \ref{lemma-2} is completed.
	\end{proof}

\begin{remark}
Note that $m=y_t+\bar m$, Lemma \ref{lemma2.2} implies
	\begin{align}\label{w-3.21-0}
		&(1+t)^{\lambda}\int_{\bf{R}}y_t^2dx
	+\int_{0}^t\int_{\bf{R}}y_t^2dxd\tau
		\leq C+C\lambda\nu_1^{\frac{\gamma(1-\lambda)}{2\lambda}}\int_{0}^t\int_{\bf{R}}(\rho^{\gamma}-\bar \rho^{\gamma})(\rho-\bar\rho)dxd\tau.
		\end{align}
\end{remark}

\begin{lemma}\label{lemma-1}
	For any  $\lambda\in(\frac{1}{\gamma},1)$, it holds that
	\begin{align}\label{w-0}
		&(1+t)^{a}\int_{\bf{R}}\eta_edx+(1-\frac{a}{2})\int_0^t\int_{\bf{R}}(1+\tau)^{a-\lambda}\frac{m^2}{\rho}dxd\tau
		\notag\\
		\leq& C
		(1+\nu_1^{1-\gamma})+\nu_1\int_0^t\int_{\bf{R}}(\rho^{\gamma}-\bar \rho^{\gamma})(\rho-\bar\rho)dxd\tau,\quad \forall \,a \in (0,1-\frac{1}{\gamma}),
	\end{align}
	where $\nu_1$ is the positive constant appeared in Lemma \ref{lemma-2} to be determined.
\end{lemma}

\begin{proof}
	Multiplying $(1+t)^{a}$ with \eqref{en.2} and integrating over $\mathbf{R}\times [0,t]$  give
	\begin{align}\label{w-1}
		&	(1+t)^{a}\int_{\bf{R}}\eta_edx+\int_0^t\int_{\bf{R}}(1+\tau)^{a-\lambda}\frac{m^2}{\rho}dxd\tau
		\notag\\
		\leq &C+	a\int_0^t\int_{\bf{R}}(1+\tau)^{a-1}\eta_edxd\tau
		\notag\\
		\leq &C+	\frac{a}{2}\int_0^t\int_{\bf{R}}(1+\tau)^{a-\lambda}\frac{m^2}{\rho}dxd\tau+\frac{a\kappa}{\gamma-1}\int_0^t\int_{\bf{R}}(1+\tau)^{a-1}\rho^{\gamma}dxd\tau.
	\end{align}
Again deviding the integral region into two parts, we have
	\begin{align}\label{w-2}
		&\int_0^t\int_{\bf{R}}(1+\tau)^{a-1}\rho^{\gamma}dxd\tau
		\notag\\
		=&\int_0^t\int_{\{x:0\leq \rho\leq \frac{(1+\tau)^{a-1}}{\nu_1}\}}(1+\tau)^{a-1}\rho^{\gamma}dxd\tau+\int_0^t\int_{\{x:\rho> \frac{(1+\tau)^{a-1}}{\nu_1}\}}(1+\tau)^{a-1}\rho^{\gamma}dxd\tau
		\notag\\
		=&: I_3+I_4.
	\end{align}
	Choosing $a<1-\frac{1}{\gamma}$ yields 
	\begin{align}
		I_3&=\int_0^t\int_{\{x:0\leq \rho\leq \frac{(1+\tau)^{a-1}}{\nu_1}\}}(1+\tau)^{a-1}\rho^{\gamma}dxd\tau
		\notag\\
		&\leq \frac{1}{\nu_1^{\gamma-1}}\int_0^t(1+\tau)^{a-1}(1+\tau)^{(a-1)(\gamma-1)}\int_{\bf{R}}\rho dxd\tau
	\leq C\nu_1^{1-\gamma}
	\end{align}
	 and
	\begin{align}\label{w-3}
		I_4&=\int_0^t\int_{\{x:\rho> \frac{(1+\tau)^{a-1}}{\nu_1}\}}(1+\tau)^{a-1}\rho^{\gamma}dxd\tau
		\notag\\
		&\leq \nu_1\int_0^t\int_{\{x:\rho> \frac{(1+\tau)^{a-1}}{\nu_1}\}}\rho^{\gamma+1}dxd\tau
		\notag\\
		&\leq  C\nu_1\int_0^t\int_{\mathbf{R}}|\rho-\bar\rho|^{\gamma+1}dxd\tau+C\nu_1\int_0^t\int_{\mathbf{R}}\bar\rho^{\gamma+1}dxd\tau
		\notag\\
		&\leq C+C\nu_1\int_0^t\int_{\mathbf{R}}|\rho-\bar\rho|^{\gamma+1}dxd\tau,
	\end{align}
	where we have used the fact $-\lambda-1+\frac{\lambda+1}{\gamma+1}<-1$ due to
	$\lambda>\frac{1}{\gamma}$.
	
	Hence, substituing \eqref{w-2}-\eqref{w-3} into \eqref{w-1} and using Lemma \ref{lemma2.3} give \eqref{w-0}. Thus, the proof of Lemma \ref{lemma-1} is completed.
\end{proof}
 
With the help of Lemma \ref{lemma-2} and Lemma \ref{lemma-1}, we are ready to estimate $\|\rho-\bar{\rho}\|_{L^{\gamma+1}}$ as follows.
\begin{lemma}\label{lemma3.5}
	For any $\gamma \in (1,3)$ and $\lambda\in(0,1)$, it holds that
	\begin{align}\label{w-3.36}
		&(1+t)^{\lambda}\int_{\mathbf{R}}(y_t^2+m^2)dx+\frac{1}{(1+t)^{\lambda}}\int_{\mathbf{R}}y^2dx+(1+t)^{\lambda}\int_{\mathbf{R}}|\rho-\bar \rho|^{\gamma+1} dx
		\notag\\
		&+\int_0^t\int_{\mathbf{R}}\frac{1}{(1+\tau)^{\lambda+1}}y^2dxd\tau+\int_0^t\int_{\mathbf{R}}(y_t^2+m^2)dxd\tau+\int_0^t\int_{\mathbf{R}}Adxd\tau
		\notag\\
		&+\int_0^t\int_{\mathbf{R}}(\rho^{\gamma}-\bar \rho^{\gamma})(\rho-\bar\rho)dxd\tau \leq C.
			\end{align}
			\end{lemma}

\begin{proof}
	Multiplying \eqref{E:2.12} with $y$ and integrating over $\mathbf{R}\times [0,t]$ give
	\begin{align}\label{w-3.31}
		&\int_{\mathbf{R}}\Big(yy_t+\frac{y^2}{2(1+t)^{\lambda}}+\bar m y-(1+t)^{\lambda}\bar \rho^{\gamma}\rho\Big)dx+\int_0^t\int_{\mathbf{R}}\frac{\lambda y^2}{2(1+\tau)^{1+\lambda}}dxd\tau
		\notag\\
		&+\kappa \int_0^t\int_{\mathbf{R}}(\rho^{\gamma}-\bar \rho^{\gamma})(\rho-\bar \rho)dxd\tau +\lambda \int_0^t\int_{\mathbf{R}}(1+\tau)^{\lambda-1}\bar \rho^{\gamma}\rho dxd\tau 
		\notag\\
		=& C_0+\int_0^t\int_{\mathbf{R}}\frac{m^2\bar \rho}{\rho}dxd\tau-\int_0^t\int_{\mathbf{R}}(1+\tau)^{\lambda}\bar \rho_t^{\gamma}\rho dxd\tau
	:=C+I_5+I_6.
		\end{align}	
If $\lambda\leq \frac{1}{\gamma}$, then we see at once from \eqref{en.2} that 
\begin{align}
	I_5&=\int_0^t\int_{\mathbf{R}}\frac{m^2\bar \rho}{\rho}dxd\tau
\leq C\int_0^t\int_{\mathbf{R}}(1+\tau)^{-\frac{\lambda+1}{\gamma+1}}\frac{m^2}{\rho}dxd\tau
	\leq C\int_0^t\int_{\mathbf{R}}\frac{m^2}{(1+\tau)^{\lambda}\rho}dxd\tau
\leq C.\nonumber
	\end{align}
If $\lambda\in(\frac{1}{\gamma},1)$, we may choose $a\in (\frac{\lambda \gamma-1}{\gamma+1},1-\frac{1}{\gamma})$ in Lemma \ref{lemma-1}  to get
	\begin{align}
		I_5&=\int_0^t\int_{\mathbf{R}}\frac{m^2\bar \rho}{\rho}dxd\tau
	\leq C\int_0^t\int_{\mathbf{R}}(1+\tau)^{-\frac{\lambda+1}{\gamma+1}}\frac{m^2}{\rho}dxd\tau
		\notag\\
		&\leq C\int_0^t\int_{\mathbf{R}}\frac{(1+\tau)^{a-\lambda}m^2}{\rho}dxd\tau
		\leq C(1+\nu_1^{1-\gamma})+C\nu_1\int_0^t\int_{\bf{R}}(\rho^{\gamma}-\bar \rho^{\gamma})(\rho-\bar\rho)dxd\tau,\nonumber
	\end{align}
	where we have used the fact that $0<\frac{\lambda \gamma-1}{\gamma+1}<1-\frac{1}{\gamma}$.
	Therefore, there always holds that 
		\begin{align}\label{w-3.32}
		I_5\leq C(1+\nu_1^{1-\gamma})+C\nu_1\int_0^t\int_{\bf{R}}(\rho^{\gamma}-\bar \rho^{\gamma})(\rho-\bar\rho)dxd\tau.
	\end{align}
On the other hand, since ${\lambda}-\frac{(\lambda+1)\gamma}{\gamma+1}<0$, we have
\begin{align}
	|I_6|=\Big|\int_0^t\int_{\mathbf{R}}(1+\tau)^{\lambda}\bar \rho_t^{\gamma}\rho dxd\tau\Big|
\leq CM\int_0^t(1+\tau)^{{\lambda}-\frac{(\lambda+1)\gamma}{\gamma+1}-1}d\tau
\leq C.
	\end{align}

Finally, we consider the nonpositive terms of left hand side of \eqref{w-3.31}. It follows from Lemma \ref{lemma2.2} that
\begin{align}\label{w-3.34}
	&\int_{\mathbf{R}}\Big(yy_t+\bar m y-(1+t)^{\lambda}\bar \rho^{\gamma}\rho\Big)dx
	\notag\\
	\leq& \frac{1}{4(1+t)^{\lambda}}\int_{\mathbf{R}}y^2dx+C(1+t)^{\lambda}\int_{\mathbf{R}}y_t^2dx+C(1+t)^{\lambda}\int_{\mathbf{R}}\bar{m}^2dx+CM(1+t)^{\lambda-\frac{(\lambda+1)\gamma}{\gamma+1}}
	\notag\\
	\leq& \frac{1}{4(1+t)^{\lambda}}\int_{\mathbf{R}}y^2dx
	+C(1+t)^{\lambda}\int_{\mathbf{R}}y_t^2dx+C.
	\end{align}
Thus, substituting \eqref{w-3.32}-\eqref{w-3.34} into \eqref{w-3.31} leads to 
	\begin{align}\label{w-3.35}
	&\frac{1}{4(1+t)^{\lambda}}\int_{\mathbf{R}}y^2dx+\int_0^t\int_{\mathbf{R}}\frac{\lambda y^2}{2(1+\tau)^{1+\lambda}}dxd\tau+\lambda \int_0^t\int_{\mathbf{R}}(1+\tau)^{\lambda-1}\bar \rho^{\gamma}\rho dxd\tau 
	\notag\\
	&+(\kappa-C\nu_1)\int_0^t\int_{\mathbf{R}}(\rho^{\gamma}-\bar \rho^{\gamma})(\rho-\bar \rho)dxd\tau 
\leq C(1+\nu_1^{1-\gamma})
	+C_3(1+t)^{\lambda}\int_{\mathbf{R}}y_t^2dx.
\end{align}
To control the last term on the right hand side of \eqref{w-3.35}, taking $\eqref{w-3.35}+\nu_2 \eqref{w-3.21}$, together with \eqref{w-3.21-0}, gives 
	\begin{align}\label{w-3.36-1}
	&\frac{1}{4(1+t)^{\lambda}}\int_{\mathbf{R}}y^2dx+(\nu_2-C_3)(1+t)^{\lambda}\int_{\bf{R}}y_t^2dx+\nu_2 (1+t)^{\lambda}\int_{\bf{R}}\tilde{\eta} dx
	\notag\\
	&+\int_0^t\int_{\mathbf{R}}\frac{\lambda y^2}{2(1+\tau)^{1+\lambda}}dxd\tau+(2-\lambda)C_2\nu_2\int_{0}^t\int_{\bf{R}}m^2dxd\tau+(3-\lambda)\nu_2\int_{0}^t\int_{\bf{R}}Adxd\tau
	\notag\\
	&+\nu_2\int_{0}^t\int_{\bf{R}}y_t^2dxd\tau+(\kappa-C\nu_1-C\lambda\nu_1^{\frac{\gamma(1-\lambda)}{2\lambda}}\nu_2)\int_0^t\int_{\mathbf{R}}(\rho^{\gamma}-\bar \rho^{\gamma})(\rho-\bar \rho)dxd\tau 
\leq C(1+\nu_1^{1-\gamma}).
\end{align}
We choose $\nu_1>0$ and $\nu_2>0$ satisfying
\begin{align*}
	\begin{cases}
		\nu_2-C_3\geq \frac{1}{2}, \\
		\kappa-C\nu_1-C\lambda\nu_1^{\frac{\gamma(1-\lambda)}{2\lambda}}\nu_2 \geq \frac{\kappa}{2},
	\end{cases}
\end{align*}
 then \eqref{w-3.36} follows from 
 \eqref{w-3.36-1}.
  Thus, the proof of Lemma \ref{lemma3.5} is completed.
	\end{proof}

\subsection{Iteration for $\|\rho-\bar{\rho}\|_{L^{\gamma+1}}$} 
As explained in introduction, the rate $(1+t)^{-\lambda}$ of $\|\rho-\bar \rho\|_{L^{\gamma+1}}$ in Lemma \ref{lemma3.5} is not fast enough to derive that of $\|\rho-\bar \rho\|_{L^{1}}$.
This subsection is devoted to improving the rate of $\|\rho-\bar{\rho}\|_{L^{\gamma+1}}$ by designing an iteration method. Define
\begin{align}\label{w-3.35-1}
\eta_{\ast}&=\tilde{\eta}-C_1\bar \rho^{\gamma+1}-C_1(\gamma+1)\bar \rho^{\gamma}(\rho-\bar \rho).
\end{align}
The entropy inequality \eqref{en.1} implies
\begin{equation}\label{3.50-1}
\eta_{\ast t}+[C_1\bar \rho^{\gamma+1}+C_1(\gamma+1)\bar \rho^{\gamma}(\rho-\bar \rho)]_t+\tilde{q}_x+\frac{2 C_2m^2}{(1+t)^{\lambda}}+\frac{A_mm}{(1+t)^{\lambda}}\leq 0.
\end{equation}
From \eqref{E:2.8-3}, we have 
\begin{equation*}
2\kappa C_2= C_1(\gamma+1).
\end{equation*}
Note that
\begin{equation}\label{3.50-2}
(\bar {\rho}^{\gamma+1})_t=(\gamma+1)\bar{\rho}^{\gamma}\bar {\rho}_t=-\frac{(\gamma+1)}{\kappa(1+t)^{\lambda}}\bar{m}^2+\{\cdots\}_x,\nonumber
\end{equation}
where $\{\cdots\}_x$  denotes terms which vanish after integrating over $\mathbf{R}$. We thus update
\eqref{3.50-1} by
\begin{align}\label{3.50-3}
&\eta_{\ast t}+ C_1(\gamma+1)[\bar \rho^{\gamma}(\rho-\bar{\rho})]_t+\{\cdots\}_x+\frac{2 C_2}{(1+t)^{\lambda}}(m-\bar m)^2
+\frac{4 C_2}{(1+t)^{\lambda}}\bar m(m-\bar m)+\frac{A_mm}{(1+t)^{\lambda}}\leq 0.
\end{align}
Since
\begin{equation}\label{3.50-4}
\frac{1}{(1+t)^{\lambda}}\bar {m}(m-\bar {m})=-\kappa(\bar {\rho}^{\gamma})_xy_t=-\kappa[(\bar {\rho}^{\gamma})_xy]_t+\kappa(\bar {\rho}^{\gamma})_t(\rho-\bar{\rho})+\{\cdots\}_x,\nonumber
\end{equation}
 we reduce \eqref{3.50-3} to 
\begin{align}\label{3.50}
&\eta_{\ast t}+\{\cdots\}_x+\frac{2C_2}{(1+t)^{\lambda}}(m-\bar m)^2+\frac{A_mm}{(1+t)^{\lambda}} \leq C_1(\gamma+1)(\bar \rho^{\gamma}_xy)_t+2C_1(\gamma+1)\bar \rho^{\gamma}_ty_x.
\end{align}


By \eqref{3.50} and the basic energy estimates \eqref{w-3.36} in Lemma \ref{lemma3.5}, we may improve the convergence rate of $\|\rho-\bar{\rho}\|_{L^{\gamma+1}}$  as follows.
\begin{lemma}\label{lemma3.6}
	For $\lambda\in(0,1)$ and $\gamma \in (1,3)$, it holds that
	\begin{align}\label{3.41}
		&(1+t)^{\mu_1(\varepsilon)}\int_{\mathbf{R}}|\rho-\bar \rho|^{\gamma+1}dx+	(1+t)^{\mu_1(\varepsilon)}\int_{\mathbf{R}}m^2dx+\int_0^t\int_{\mathbf{R}}(1+\tau)^{\mu_1(\varepsilon)-\lambda}y_t^2dxd\tau
		\notag\\
		&+\int_0^t\int_{\mathbf{R}}(1+\tau)^{\mu_1(\varepsilon)-\lambda}Adxd\tau\leq C,
		\end{align}
	where for any small $\varepsilon>0$,
	\begin{equation}\label{newq}
	\mu_1(\varepsilon)=\min\{1, 1+\frac{\lambda}{2}-\frac{\lambda+1}{2(\gamma+1)}\}-\varepsilon:=\tilde {\mu}_1-\varepsilon.
	\end{equation}
	\end{lemma}

\begin{proof}
	Multiplying $(1+t)^{\mu_1}$ with \eqref{3.50} and integrating over $\mathbf{R}\times[0,t]$ give
\begin{align}\label{3.54}
&(1+t)^{\mu_1}\int_{\mathbf{R}}\eta_{\ast}dx+2C_2\int_0^t\int_{\mathbf{R}}(1+\tau)^{\mu_1-\lambda}(m-\bar m)^2dxd\tau
+\int_0^t\int_{\mathbf{R}}(1+\tau)^{\mu_1-\lambda}A_mmdxd\tau 
\notag\\
\leq&C+C(1+t)^{\mu_1}\int_{\mathbf{R}}\bar \rho^{\gamma}_xydx+\mu_1\int_0^t\int_{\mathbf{R}}(1+\tau)^{\mu_1-1}\eta_{\ast}dxd\tau
\notag\\
&+C_1(\gamma+1) \mu_1\int_0^t\int_{\mathbf{R}}(1+\tau)^{\mu_1-1}\bar \rho^{\gamma}_xydxd\tau+2C_1(\gamma+1) \int_0^t\int_{\mathbf{R}}(1+\tau)^{\mu_1}\bar \rho^{\gamma}_ty_xdxd\tau
\notag\\
=&:C+I_7+I_8+I_9+I_{10}.
\end{align}	
Then it follows from Lemma \ref{lemma3.5} and \eqref{newq} that 
\begin{align}\label{3.55}
I_7&=(1+t)^{\mu_{1}}\int_{\mathbf{R}}\bar \rho^{\gamma}_xydx
\leq \frac{1}{(1+t)^{\lambda}}\int_{\mathbf{R}}y^2dx+(1+t)^{\lambda+2\mu_1}\int_{\mathbf{R}}(\bar \rho^{\gamma}_x)^2dx
\notag\\
&\leq C+(1+t)^{2\mu_1-\lambda-2+\frac{\lambda+1}{\gamma+1}}
\leq C,
\end{align}
\begin{align}
	I_9&=C_1(\gamma+1) \mu_1\int_0^t\int_{\mathbf{R}}(1+\tau)^{\mu_1-1}\bar \rho^{\gamma}_xydxd\tau
	\notag\\
	&\leq C\mu_1\int_0^t\Big(\int_{\mathbf{R}}\frac{y^2}{(1+\tau)^{\lambda}}dx\Big)^{\frac{1}{2}}\Big(\int_{\mathbf{R}}(1+\tau)^{2\mu_1-2+\lambda}(\bar \rho^{\gamma}_x)^2dx\Big)^{\frac{1}{2}}d\tau
	\notag\\
	&\leq C\int_0^t(1+\tau)^{\mu_1-2-\frac{\lambda}{2}+\frac{(\lambda+1)\gamma}{2(\gamma+1)}}d\tau
\leq C
\end{align}
and 
\begin{align}\label{3.57}
	I_{10}&=2C_1(\gamma+1) \int_0^t\int_{\mathbf{R}}(1+\tau)^{\mu_1}\bar \rho^{\gamma}_ty_xdxd\tau
	\notag\\
	&\leq C\int_0^t\int_{\mathbf{R}}\bar \rho^{\gamma-1}(\rho-\bar \rho)^2dxd\tau+C\int_0^t\int_{\mathbf{R}}(1+\tau)^{2\mu_1}(\bar \rho^{\gamma}_t)^2\bar \rho^{1-\gamma}dxd\tau
	\notag\\
	&\leq C+C\int_0^t(1+\tau)^{2\mu_1-\lambda-3+\frac{\lambda+1}{\gamma+1}}d\tau
\leq C.
\end{align}
Note $\mu_1\leq 1$, we conclude from Lemma \ref{lemma2.3} and Lemma \ref{lemma3.5} that 
\begin{align}\label{3.591}
	I_{8}&=\mu_1\int_0^t\int_{\mathbf{R}}(1+\tau)^{\mu_1-1}\eta_{\ast}dxd\tau
	\notag\\
	&= C_1\mu_1\int_0^t\int_{\mathbf{R}}(1+\tau)^{\mu_1-1}(\rho^{\gamma+1}-\bar \rho^{\gamma+1}-(\gamma+1)\bar \rho^{\gamma}(\rho-\bar \rho))dxd\tau
	\notag\\
	&\quad +C_2\mu_1\int_0^t\int_{\mathbf{R}}(1+\tau)^{\mu_1-1}m^2dxd\tau+\mu_1\int_0^t\int_{\mathbf{R}}(1+\tau)^{\mu_1-1}Adxd\tau
	\notag\\
	&\leq C+\mu_1\int_0^t\int_{\mathbf{R}}(1+\tau)^{\mu_1-1}Adxd\tau.
\end{align}
Substituting \eqref{3.55}-\eqref{3.57} into \eqref{3.54} leads to
\begin{align*}
	&(1+t)^{\mu_1(\varepsilon)}\int_{\mathbf{R}}\eta_{\ast}dx+2C_2\int_0^t\int_{\mathbf{R}}(1+\tau)^{\mu_1(\varepsilon)-\lambda}y_t^2dxd\tau
+(3-\mu_1(\varepsilon))\int_0^t\int_{\mathbf{R}}(1+\tau)^{\mu_1(\varepsilon)-\lambda}Adxd\tau \leq C,
\end{align*}	
then \eqref{3.41} follows from the definition of $\eta_{\ast}$ and Lemma \ref{lemma2.3}, and the proof of Lemma \ref{lemma3.6} is completed.
	\end{proof}
\begin{remark}
	Since $\tilde {\mu}_1-\lambda>0$, the rate of $\|\rho-\bar \rho\|_{L^{\gamma+1}}$ is improved and  the better estimates for $(1+t)^{\mu_1}\int_{\mathbf{R}}m^2dx$ and $\int_0^t\int_{\mathbf{R}}(1+\tau)^{\mu_1-\lambda}y_t^2dxd\tau$ are obtained.	
	\end{remark}
 
Based on the new estimates for $(1+t)^{\mu_1}\int_{\mathbf{R}}m^2dx$ and $\int_0^t\int_{\mathbf{R}}(1+\tau)^{\mu_1-\lambda}y_t^2dxd\tau$, we can   study $\int_0^t\int_{\mathbf{R}}(1+\tau)^{\theta_1}|\rho-\bar \rho|^{\gamma+1}dxd\tau $ as in Lemma \ref{w-lemma3.7} below.

\begin{lemma}\label{bulemma3.7}
	For $\lambda \in (0,1)$ and any given $\beta>0$, it holds that 
	\begin{align}\label{3.499}
		&(1+t)^{\theta}\int_{\bf{R}}\eta_edx+(1-\frac{\theta}{2})\int_{0}^t\int_{\bf{R}}(1+\tau)^{\theta-\lambda}\frac{m^2}{\rho}dxd\tau
		\leq C+\nu_3\int_{0}^t\int_{\bf{R}}(1+\tau)^{\beta}|\rho-\bar \rho|^{\gamma+1}dxd\tau,
		\end{align}
	where $\theta<\min\{\frac{(\gamma-1)(\lambda+1)}{\gamma+1},\frac{(\gamma-1)(1+\beta)}{\gamma}\}$ and $\nu_3$ is a constant to be determined.
	\end{lemma}

\begin{proof}
		Multiplying \eqref{en.2} with $(1+t)^{\theta}$ and integrating over $\mathbf{R}\times [0,t]$ give
			\begin{align}\label{3.500}
			&(1+t)^{\theta}\int_{\bf{R}}\eta_edx+(1-\frac{\theta}{2})\int_{0}^t\int_{\bf{R}}(1+\tau)^{\theta-\lambda}\frac{m^2}{\rho}dxd\tau
			\notag\\
			\leq & C+C\int_{0}^t\int_{\bf{R}}(1+\tau)^{\theta-1}|\rho-\bar \rho|^{\gamma}dxd\tau+C\int_{0}^t\int_{\bf{R}}(1+\tau)^{\theta-1}\bar \rho^{\gamma}dxd\tau.
		\end{align}
	Noting that $\theta<\min\{\frac{(\gamma-1)(\lambda+1)}{\gamma+1},\frac{(\gamma-1)(1+\beta)}{\gamma}\}$, we get 
	\[\beta+(\theta-1-\beta)\gamma<-1\quad \mbox{and} \quad \theta-1-\frac{(\lambda+1)(\gamma-1)}{\gamma+1}<-1,\]
	which yields that 
	\begin{align}\label{3.511}
		&\int_{0}^t\int_{\bf{R}}(1+\tau)^{\theta-1}|\rho-\bar \rho|^{\gamma}dxd\tau
		\notag\\
		\leq& \nu_3\int_{0}^t\int_{\bf{R}}(1+\tau)^{\beta}|\rho-\bar \rho|^{\gamma+1}dxd\tau+C\int_{0}^t\int_{\bf{R}}(1+\tau)^{\beta+(\theta-1-\beta)\gamma}|\rho-\bar \rho|dxd\tau
			\notag\\
		\leq& C+\nu_3\int_{0}^t\int_{\bf{R}}(1+\tau)^{\beta}|\rho-\bar \rho|^{\gamma+1}dxd\tau
		\end{align}
	and
	\begin{align}\label{3.522}
		\int_{0}^t\int_{\bf{R}}(1+\tau)^{\theta-1}\bar \rho^{\gamma}dxd\tau\leq M\int_{0}^t(1+\tau)^{\theta-1-\frac{(\lambda+1)(\gamma-1)}{\gamma+1}}d\tau\leq C,
		\end{align} 
	where $\nu_3$ is a constant to be determined.
	Substituting \eqref{3.511}-\eqref{3.522} into \eqref{3.500} gives \eqref{3.499}, and the proof of Lemma \ref{bulemma3.7} is completed.
	\end{proof}

\begin{lemma}\label{w-lemma3.7}
	For $\lambda\in(0,1)$ and $\gamma \in (1,3)$, it holds that
	\begin{align}\label{3.49}
		&(1+t)^{\theta_1(\varepsilon)-\lambda}\int_{\mathbf{R}}y^2dx+\int_0^t\int_{\mathbf{R}}(1+\tau)^{\theta_1(\varepsilon)}(\rho^{\gamma}-\bar \rho^{\gamma})(\rho-\bar \rho)dxd\tau+\int_0^t\int_{\mathbf{R}}(1+\tau)^{\theta_1(\varepsilon)}m^2dxd\tau
		\notag\\
		& +(\lambda-\theta_1(\varepsilon))\int_0^t\int_{\mathbf{R}}(1+\tau)^{\theta_1(\varepsilon)-\lambda-1}y^2dxd\tau
		\leq C,
	\end{align}
	where $\theta_1(\varepsilon)=\min\{\tilde {\mu}_1-\lambda, \lambda, \frac{\gamma-\lambda}{\gamma+1}\}-\varepsilon=:\tilde \theta_1-\varepsilon$ for small $\varepsilon>0$. 
\end{lemma}
\begin{proof}
		Multiplying \eqref{E:2.12} with $(1+t)^{\theta_1}y$ and integrating over $\mathbf{R}\times [0,t]$ give
	\begin{align}\label{w-3.50}
		&\frac{1 }{2}(1+t)^{\theta_1-\lambda}\int_{\mathbf{R}}y^2dx+\int_0^t\int_{\mathbf{R}}(1+\tau)^{\theta_1}m^2 dxd\tau 
		\notag\\
		&+\frac{1}{2}(\lambda-\theta_1)\int_0^t\int_{\mathbf{R}}(1+\tau)^{\theta_1-\lambda-1} y^2dxd\tau+\kappa \int_0^t\int_{\mathbf{R}}(1+\tau)^{\theta_1}(\rho^{\gamma}-\bar \rho^{\gamma})(\rho-\bar \rho)dxd\tau
		\notag\\
		=& C_0+\Big[(1+t)^{\theta_1}\int_{\mathbf{R}}ymdx+\int_0^t\int_{\mathbf{R}}(1+\tau)^{\theta_1}y_t^2 dxd\tau\Big]+\theta_1\int_0^t\int_{\mathbf{R}}(1+\tau)^{\theta_1-1}yy_t dxd\tau 
		\notag\\
		&+\int_0^t\int_{\mathbf{R}}(1+\tau)^{\theta_1}\bar my_t dxd\tau
		+\theta_1\int_0^t\int_{\mathbf{R}}(1+\tau)^{\theta_1-1}\bar my dxd\tau  
		\notag\\
		&+\int_0^t\int_{\mathbf{R}}(1+\tau)^{\theta_1}\frac{m^2}{\rho}\bar \rho dxd\tau
	:=C_0+I_{11}+I_{12}+I_{13}+I_{14}+I_{15}.
	\end{align}
Then, we may deduce from 
 Lemma \ref{lemma3.6} and $\lambda\in(0,1)$ that
\begin{align}\label{3.51}
	I_{11}&=(1+t)^{\theta_1}\int_{\mathbf{R}}ymdx+\int_0^t\int_{\mathbf{R}}(1+\tau)^{\theta_1}y_t^2 dxd\tau
	\notag\\
&\leq \frac{1 }{4}(1+t)^{\theta_1-\lambda}\int_{\mathbf{R}}y^2dx+C\int_{\mathbf{R}}(1+t)^{\theta_1+\lambda}m^2dx+C
\notag\\
&\leq  \frac{1 }{4}(1+t)^{\theta_1-\lambda}\int_{\mathbf{R}}y^2dx+C,\\
	I_{12}&=\theta_1\int_0^t\int_{\mathbf{R}}(1+\tau)^{\theta_1-1}yy_t dxd\tau 
	\notag\\
	&\leq \nu_4\int_0^t\int_{\mathbf{R}}(1+\tau)^{\theta_1-\lambda-1} y^2dxd\tau+C\int_0^t\int_{\mathbf{R}}(1+\tau)^{\theta_1+\lambda-1}y_t^2dxd\tau
	\notag\\
	&\leq \nu_4\int_0^t\int_{\mathbf{R}}(1+\tau)^{\theta_1-\lambda-1} y^2dxd\tau+C,
	\end{align}
and 
\begin{align}\label{3.611}
	I_{13}&=\int_0^t\int_{\mathbf{R}}(1+\tau)^{\theta_1}\bar my_t dxd\tau 
	\notag\\
	&\leq \nu_4\int_0^t\int_{\mathbf{R}}(1+\tau)^{\theta_1}m^2 dxd\tau+C(\nu_4)\int_0^t\int_{\mathbf{R}}(1+\tau)^{\theta_1}y_t^2dxd\tau
	\notag\\
	&\leq \nu_4\int_0^t\int_{\mathbf{R}}(1+\tau)^{\theta_1}m^2 dxd\tau+C,
\end{align}
where $\nu_4>0$ is a constant to be determined.

Next, we will handle the remained terms $I_{14}$ and $I_{15}$. Note that 
\[\theta_1+(\lambda-1)(1+\frac{1}{\gamma})-\frac{(\lambda+1)\gamma}{\gamma+1}<-1,\]
we have
\begin{align}\label{3.63}
	I_{14}&=\theta_1\int_0^t\int_{\mathbf{R}}(1+\tau)^{\theta_1-1}\bar my dxd\tau
	\notag\\
	&\leq \nu_4 \int_0^t\int_{\mathbf{R}}(1+\tau)^{\theta_1}|\rho-\bar \rho|^{\gamma+1}dxd\tau+C\int_0^t(1+\tau)^{\theta_1+(\lambda-1)(1+\frac{1}{\gamma})-\frac{(\lambda+1)\gamma}{\gamma+1}}d\tau
	\notag\\
	&\leq C+\nu_4 \int_0^t\int_{\mathbf{R}}(1+\tau)^{\theta_1}|\rho-\bar \rho|^{\gamma+1}dxd\tau.
\end{align}
On the other hand, note that 
\begin{equation}\label{3.64}
\theta_1-\frac{\lambda+1}{\gamma+1}<\min\{\frac{(\gamma-1)(\lambda+1)}{\gamma+1},\frac{(\gamma-1)(1+\theta_1)}{\gamma}\}-\lambda.
\end{equation}
Set $\beta=\theta_1$ in Lemma 
\ref{bulemma3.7}, it follows from \eqref{3.499} and \eqref{3.64} that
\begin{align} \label{3.65}
	I_{15}&=\int_0^t\int_{\mathbf{R}}(1+\tau)^{\theta_1}\frac{m^2}{\rho}\bar \rho dxd\tau
\leq C\int_0^t\int_{\mathbf{R}}(1+\tau)^{\theta_1-\frac{\lambda+1}{\gamma+1}}\frac{m^2}{\rho} dxd\tau
	\notag\\
	&\leq C+\nu_3\int_{0}^t\int_{\bf{R}}(1+\tau)^{\theta_1}|\rho-\bar \rho|^{\gamma+1}dxd\tau.
	\end{align}
Substituting \eqref{3.51}-\eqref{3.63} and \eqref{3.65} into \eqref{w-3.50}, together with Lemma \ref{lemma2.3}, yields that
\begin{align*}
		&\frac{1 }{4}(1+t)^{\theta_1-\lambda}\int_{\mathbf{R}}y^2dx+(1-\nu_4)\int_0^t\int_{\mathbf{R}}(1+\tau)^{\theta_1}m^2 dxd\tau 
	\notag\\
	&+\frac{1}{2}(1-\nu_4)(\lambda-\theta_1)\int_0^t\int_{\mathbf{R}}(1+\tau)^{\theta_1-\lambda-1} y^2dxd\tau
	\notag\\
	&+[\kappa-C(\nu_3+\nu_4)]\int_0^t\int_{\mathbf{R}}(1+\tau)^{\theta_1}(\rho^{\gamma}-\bar \rho^{\gamma})(\rho-\bar \rho)dxd\tau
\leq C.
	\end{align*}
Choosing 
\begin{align*}
	\begin{cases}
		1-\nu_4>\frac{1}{2}, \\
		\kappa-C(\nu_3+\nu_4)>\frac{\kappa}{2},
	\end{cases}
\end{align*}
 leads to \eqref{3.49}. Thus the proof of Lemma \ref{w-lemma3.7} is completed.
	\end{proof}
	
 With the aid of lemma \ref{lemma3.6} and lemma \ref{w-lemma3.7}, we can start the iteration procedure.
 

\begin{lemma}\label{lemma3.8-1}
	Under the conditions of Theorem \ref{theorem 1.1}, it holds that
	\begin{align}\label{3.58-1}
		&(1+t)^{\mu_{k+1}(\varepsilon)}\int_{\mathbf{R}}|\rho-\bar \rho|^{\gamma+1}dx+	(1+t)^{\mu_{k+1}(\varepsilon)}\int_{\mathbf{R}}m^2dx+\int_0^t\int_{\mathbf{R}}(1+\tau)^{\mu_{k+1}(\varepsilon)-\lambda}y_t^2dxd\tau
		\notag\\
		&+\int_0^t\int_{\mathbf{R}}(1+\tau)^{\mu_{k+1}(\varepsilon)-\lambda}Adxd\tau\leq C,
	\end{align}
	where 
\begin{equation}\label{3.691-1}
\mu_{k+1}(\varepsilon)=\tilde {\mu}_{k+1}-\varepsilon=\min\{1+\tilde \theta_k,\, 1+\frac{\lambda}{2}-\frac{\lambda+1}{2(\gamma+1)}+\frac{\tilde \theta_k}{2} \}-\varepsilon, \quad \forall\, k\in \mathbf{N}
\end{equation}
and 
\begin{equation}\label{3.691-2}
\tilde \theta_k=\min\{\tilde {\mu}_k-\lambda, \lambda,\frac{\gamma-\lambda}{\gamma+1}\}
\end{equation}
are increasing sequences.
\end{lemma}
\begin{proof}
From Lemma \ref{lemma3.6}, \eqref{3.58-1} holds for $k=0$ with the supplementary definition of $\tilde \theta_0=0$. Next we will use Lemma \ref{w-lemma3.7} to show that  \eqref{3.58-1} is true for $k=1$.	Multiplying $(1+t)^{\mu_2}$ with \eqref{3.50} and integrating over $\mathbf{R}\times[0,t]$, we have 
	\begin{align}\label{w-3.60}
		&(1+t)^{\mu_2}\int_{\mathbf{R}}\eta_{\ast}dx+2C_2\int_0^t\int_{\mathbf{R}}(1+\tau)^{\mu_2-\lambda}y_t^2dxd\tau+3\int_0^t\int_{\mathbf{R}}(1+\tau)^{\mu_2-\lambda}Adxd\tau 
		\notag\\
		\leq&C+C(1+t)^{\mu_2}\int_{\mathbf{R}}\bar \rho^{\gamma}_xydx+\mu_2\int_0^t\int_{\mathbf{R}}(1+\tau)^{\mu_2-1}\eta_{\ast}dxd\tau
		\notag\\
		&+C_1(\gamma+1)\mu_2\int_0^t\int_{\mathbf{R}}(1+\tau)^{\mu_2-1}\bar \rho^{\gamma}_xydxd\tau+2C_1(\gamma+1)\int_0^t\int_{\mathbf{R}}(1+\tau)^{\mu_2}\bar \rho^{\gamma}_ty_xdxd\tau
		\notag\\
		=&:C+I_{16}+I_{17}+I_{18}+I_{19}.
	\end{align}
	Following the same argument  in \eqref{3.41},  with $(1+t)^{-\lambda}\int_{\mathbf{R}}y^2dx$ and $\int_0^t\int_{\mathbf{R}}\bar \rho^{\gamma-1}(\rho-\bar \rho)^2dxd\tau$ replaced by $(1+t)^{\theta_1-\lambda}\int_{\mathbf{R}}y^2dx$ and $\int_0^t\int_{\mathbf{R}}(1+\tau)^{\theta_1}\bar \rho^{\gamma-1}(\rho-\bar \rho)^2dxd\tau$ in $I_{16},\,I_{18}$ and $I_{19}$ respectively, we obtain 
	\begin{align}\label{3.61}
		I_{16}+I_{18}+I_{19}&\leq C+C(1+t)^{2\mu_2-\lambda-2+\frac{\lambda+1}{\gamma+1}-\theta_1}\leq C,
		\end{align}
	since 
	\[\mu_2<1+\frac{\lambda}{2}-\frac{\lambda+1}{2(\gamma+1)}+\frac{\theta_1}{2}=1+\frac{\lambda}{2}-\frac{\lambda+1}{2(\gamma+1)}+\frac{\tilde \theta_1}{2}-\frac{\varepsilon}{2}.\]
Moreover, with $\int_0^t\int_{\mathbf{R}}m^2dxd\tau$ replaced by
$\int_0^t\int_{\mathbf{R}}(1+\tau)^{\theta_1}m^2dxd\tau$ in $I_{17}$, we get
\begin{align}\label{3.62}
	I_{17}\leq C+\mu_2\int_0^t\int_{\mathbf{R}}(1+\tau)^{\mu_2-1}Adxd\tau.
\end{align}	
Thus, substituting \eqref{3.61}-\eqref{3.62} into \eqref{w-3.60} yields
	\begin{align}\label{3.59}
	&(1+t)^{\mu_2(\varepsilon)}\int_{\mathbf{R}}|\rho-\bar \rho|^{\gamma+1}dx+	(1+t)^{\mu_2(\varepsilon)}\int_{\mathbf{R}}m^2dx+\int_0^t\int_{\mathbf{R}}(1+\tau)^{\mu_2(\varepsilon)-\lambda}Adxd\tau
	\notag\\
	&+\int_0^t\int_{\mathbf{R}}(1+\tau)^{\mu_2(\varepsilon)-\lambda}y_t^2dxd\tau\leq C.
\end{align}

From \eqref{3.59} and the same line in the proof of Lemma \ref{w-lemma3.7} with $\theta_1$ and $\mu_1$ replaced by $\theta_2$ and $\mu_2$,  the decay rate \eqref{3.49} in Lemma \ref{w-lemma3.7} can be improved to 
\begin{align}\label{3.59-1}
		&(1+t)^{\theta_2(\varepsilon)-\lambda}\int_{\mathbf{R}}y^2dx+\int_0^t\int_{\mathbf{R}}(1+\tau)^{\theta_2(\varepsilon)}(\rho^{\gamma}-\bar \rho^{\gamma})(\rho-\bar \rho)dxd\tau
		\notag\\
		&+\int_0^t\int_{\mathbf{R}}(1+\tau)^{\theta_2(\varepsilon)}m^2dxd\tau+(\lambda-\theta_2(\varepsilon))\int_0^t\int_{\mathbf{R}}(1+\tau)^{\theta_2(\varepsilon)-\lambda-1}y^2dxd\tau
		\leq C.
	\end{align}
	Repeating the same argument for  $k=2,3,\cdots,$ we obtain  \eqref{3.58-1} for increasing sequences $\tilde \mu_{k+1}$ and $\tilde \theta_k$.
\end {proof}

\begin{lemma}\label{lemma3.8}
	Under the conditions of Theorem \ref{theorem 1.1}, it holds that
	\begin{align}\label{3.58}
		&(1+t)^{\mu(\varepsilon)}\int_{\mathbf{R}}|\rho-\bar \rho|^{\gamma+1}dx+	(1+t)^{\mu(\varepsilon)}\int_{\mathbf{R}}m^2dx+2C_2\int_0^t\int_{\mathbf{R}}(1+\tau)^{\mu(\varepsilon)-\lambda}y_t^2dxd\tau
		\notag\\
		&+\int_0^t\int_{\mathbf{R}}(1+\tau)^{\mu(\varepsilon)-\lambda}Adxd\tau\leq C,
	\end{align}
	where 
\begin{equation}\label{3.691}
\mu(\varepsilon)=
\left\{
\begin{array}{ll}
1+\lambda-\frac{\lambda+1}{2(\gamma+1)}-\varepsilon, &\quad \lambda\in(0,\frac{\gamma}{\gamma+2}],\\
\frac{3}{2}+\frac{\lambda}{2}-\frac{\lambda+1}{\gamma+1}-\varepsilon, &\quad \lambda \in [\frac{\gamma}{\gamma+2},1).
\end{array}
\right.
\end{equation}
\end{lemma}
\begin{proof}
We only need to consider the case of $\lambda\in (0,\frac{\gamma}{\gamma+2}]$ since the other case $\lambda\in [\frac{\gamma}{\gamma+2},1)$ can be treated in the similar way.
Note that   from \eqref{3.691-1} and \eqref{3.691-2}, we have 
\begin{align}\label{3.71-1}
\tilde {\mu}_{k+1}
&=
\left\{
\begin{array}{ll}
1+\tilde \theta_k, &\quad \tilde \theta_k \in (-\infty, \frac{\lambda\gamma-1}{\gamma+1}],\\
1+\frac{\lambda}{2}-\frac{\lambda+1}{2(\gamma+1)}+\frac{\tilde \theta_k}{2}, &\quad \tilde \theta_k \in (\frac{\lambda\gamma-1}{\gamma+1},\infty),
\end{array}
\right.
\end{align}
and 
\begin{equation}\label{3.791}
\tilde \theta_k=\min\{\tilde {\mu}_k-\lambda, \lambda\}, \quad \lambda\in (0,\frac{\gamma}{\gamma+2}].
\end{equation}
We claim that there exists $k_1\in \mathbf{N}$ such that for any $k\geq k_1$,
\begin{align}\label{3.74}
	\tilde \theta_{k}\geq \tilde \theta_{k_1}
	>\frac{\lambda\gamma-1}{\gamma+1}.
\end{align}
If \eqref{3.74} does not hold, 
 it follows from \eqref{3.791} that
\begin{equation}\label{ne3.71}
0<\tilde \theta_k=\tilde {\mu}_k-\lambda\leq \frac{\lambda\gamma-1}{\gamma+1}<\lambda, \quad \forall \,k \in \mathbf{N},
\end{equation}
which, together with \eqref{3.71-1}, indicates that
\[\tilde {\mu}_{k+1}=1+\tilde \theta_k=1+\tilde {\mu}_k-\lambda=k(1-\lambda)+\tilde {\mu}_1\rightarrow +\infty \quad \mbox{as}\quad k \rightarrow +\infty.\]
This contradicts \eqref{ne3.71}, and then \eqref{3.74} holds. 	We also claim that there exists $k_2\geq k_1$ such that 	
\begin{equation}\label{3.70}
\tilde \mu_{k_2}-\lambda \geq\lambda.
\end{equation}
If not, then 
\begin{equation}\label{3.77}
\tilde \theta_k=\tilde {\mu}_k-\lambda<\lambda,\quad \forall \,k\geq k_1,
\end{equation}
 which, together with \eqref{3.71-1} and \eqref{3.74}, yields that
\begin{align}\label{3.771}
	\tilde {\mu}_{k+1}-\lambda&=1+\frac{\lambda}{2}-\frac{\lambda+1}{2(\gamma+1)}+\frac{1}{2}(\tilde {\mu}_k-\lambda)-\lambda
	\notag\\
	&=\Big(1+\frac{1}{2}\Big)\Big(1+\frac{\lambda}{2}-\frac{\lambda+1}{2(\gamma+1)}-\lambda\Big)+\frac{1}{2^{2}}(\tilde {\mu}_{k-1}-\lambda)
		\notag\\
	&=\Big(1+\frac{1}{2}+\cdots+\frac{1}{2^{k-k_1}}\Big)\Big(1+\frac{\lambda}{2}-\frac{\lambda+1}{2(\gamma+1)}-\lambda\Big)+\frac{1}{2^{k+1-k_1}}(\tilde {\mu}_{k_1}-\lambda)
		\notag\\
	&>\Big(1+\frac{1}{2}+\cdots+\frac{1}{2^{k-k_1}}\Big)\Big(1-\frac{\lambda}{2}-\frac{\lambda+1}{2(\gamma+1)}\Big).
	\end{align}
Note that $1-\frac{\lambda}{2}-\frac{\lambda+1}{2(\gamma+1)}
>\frac{\lambda}{2}$ since $\lambda\leq \frac{\gamma}{\gamma+2}<\frac{2\gamma+1}{2\gamma+3}$. Thus, there always exists sufficiently large $k_3$ such that $	\tilde {\mu}_{k_3+1}-\lambda> \lambda$, which contradicts \eqref{3.77}, and then \eqref{3.70} holds. 
Finally, it follows from \eqref{3.70} 
that 
\begin{equation}\label{3.71-3}
\tilde {\mu}_{k_1+1}=
1+\lambda-\frac{\lambda+1}{2(\gamma+1)}.
\end{equation}
 Thus, the proof of Lemma \ref{lemma3.8} is completed.
	\end{proof}

\subsection{Decay rate of $\|\rho-\bar{\rho}\|_{L^1}$} We are ready to derive the convergence rate of $\|\rho-\bar{\rho}\|_{L^1}$ since the rate of $\|\rho-\bar{\rho}\|_{L^{\gamma+1}}$ obtained in Lemma \ref{lemma3.8} is fast enough.
Following  \cite{Huang-Pan-Wang}, we have 
\begin{lemma} \label{lemma3.7} 	For $\lambda\in(0,1)$ and $\gamma \in (1,3)$, it holds that
	\begin{align*}
	\int_{\mathbf{R}}|\rho-\bar\rho|dx\leq C(1+t)^{-\alpha(\varepsilon)},
	\end{align*}
where $\alpha(\varepsilon)$ satisfies:\\
	\begin{equation}\label{3.89}
\alpha(\varepsilon)=
\left\{
\begin{array}{ll}
\frac{\lambda+1}{4(\gamma+1)}-\varepsilon, &\quad \lambda\in(0,\frac{\gamma}{\gamma+2}]\\
\frac{1-\lambda}{4}-\varepsilon, &\quad \lambda \in [\frac{\gamma}{\gamma+2},1).\\
\end{array}
\right.
\end{equation}
\end{lemma}
\begin{proof}
Following  \cite{Huang-Pan-Wang}, we divide the support subset of $\bar \rho$ into two parts:
	\begin{align}
	\Omega_0&=\Big(-\sqrt{\frac{A}{B}}(1+t)^{\frac{\lambda+1}{\gamma+1}}+\sqrt{\frac{A}{B}}(1+t)^{-b},\sqrt{\frac{A}{B}}(1+t)^{\frac{\lambda+1}{\gamma+1}}-\sqrt{\frac{A}{B}}(1+t)^{-b}\Big),\nonumber\\
		\Omega_1&=\Big(-\sqrt{\frac{A}{B}}(1+t)^{\frac{\lambda+1}{\gamma+1}}, -\sqrt{\frac{A}{B}}(1+t)^{\frac{\lambda+1}{\gamma+1}}+\sqrt{\frac{A}{B}}(1+t)^{-b}\Big)
		\notag\\
		&\quad \,\,\cup\Big(\sqrt{\frac{A}{B}}(1+t)^{\frac{\lambda+1}{\gamma+1}}-\sqrt{\frac{A}{B}}(1+t)^{-b}, \sqrt{\frac{A}{B}}(1+t)^{\frac{\lambda+1}{\gamma+1}}\Big),\nonumber
	\end{align}
	where $b>1$ is a constant. Naturally, $\{\bar \rho>0\}=\Omega_0\cup \Omega_1$.
	Then, we deduce from Lemma \ref{lemma3.8} that
	\begin{align*}
	\int_{\mathbf{R}}|\rho-\bar \rho|dx &\leq 2\int_{\bar \rho>0}|\rho-\bar 
	\rho|dx
	\notag\\
	&=2\int_{\Omega_0}|\rho-\bar 
	\rho|dx+2\int_{\Omega_1}|\rho-\bar 
	\rho|dx
	\notag\\
	&\leq C\Big(\int_{\Omega_0}\bar \rho^{\gamma-1}|\rho-\bar 
	\rho|^{2}dx\Big)^{\frac{1}{2}}\Big(\int_{\Omega_0}\bar \rho^{1-\gamma}dx\Big)^{\frac{1}{2}}+C\int_{\Omega_1}|\rho-\bar 
	\rho|dx
	\notag\\
	&\leq C(1+t)^{-\frac{1}{2}(\mu(\varepsilon)-\frac{\gamma(\lambda+1)}{\gamma+1})}\sqrt{\ln(1+t)},
	\end{align*}
where we have used the fact that
\begin{align*}
\int_{\Omega_0}\bar \rho^{1-\gamma}dx&\leq (1+t)^{\frac{\gamma(\lambda+1)}{\gamma+1}}\int_{0}^{\sqrt{\frac{A}{B}}(1-(1+t)^{-b-\frac{\lambda+1}{\gamma+1}})}(A-B\xi^2)^{-1}d\xi
\notag\\
&\leq C(1+t)^{\frac{\gamma(\lambda+1)}{\gamma+1}}\ln(1+t).
\end{align*}
It follows from \eqref{3.691} that
\begin{equation*}
\mu(\varepsilon)-\frac{\gamma(\lambda+1)}{\gamma+1}=
\left\{
\begin{array}{ll}
\frac{\lambda+1}{2(\gamma+1)}-\varepsilon, &\quad \lambda\in(0,\frac{\gamma}{\gamma+2}],\\
\frac{1-\lambda}{2}-\varepsilon, &\quad \lambda \in [\frac{\gamma}{\gamma+2},1),
\end{array}
\right.
\end{equation*}
which leads to
\begin{equation*}
(1+t)^{-\frac{1}{2}(\mu(\varepsilon)-\frac{\gamma(\lambda+1)}{\gamma+1})}\sqrt{\ln(1+t)}\leq (1+t)^{-\alpha(\varepsilon)},
\end{equation*}
where $\alpha(\varepsilon)$ is given in \eqref{3.89}.
Thus, the proof of Lemma \ref{lemma3.7} is completed.
	\end{proof}

{\bf \hspace{-1em}Proof of Theorem \ref{theorem 1.1}}
Theorem \ref{theorem 1.1} holds directly from Lemma \ref{lemma3.8} and Lemma \ref{lemma3.7}. 
$\hspace{7em}\Box$

\section{Acknowledgments}
S. Geng's research is supported in part by the National Natural Science Foundation of China (No. 11701489), by Natural Science Foundation of Hunan Province of China (No. 2018JJ2373), and by Excellent Youth Project of Hunan Education Department (No.18B054).
 F. Huang's research is supported in part by the
National Natural Science Foundation of China (No. 11371349).

     \end{document}